\documentclass{amsart}
\usepackage{amsfonts}

\newtheorem{theorem}{Theorem}
\theoremstyle{plain}

\newtheorem{definition}{Definition}

\newtheorem{lemma}{Lemma}

\numberwithin{equation}{section}
\input{tcilatex}

\begin{document}
\title[Everywhere h\"{o}lder continuity]{On the local everywhere H\"{o}lder
continuity of the minima of a class of vectorial integral functionals of the
calculus of variations}
\author{Tiziano Granucci}
\address{ISIS\ Leonardo da Vinci, Firenze, via del Terzolle 91, 50100, Italy}
\email{tizianogranucci@libero.it, tizianogranucci@isisdavinci.eu}
\urladdr{https://www.tizianogranucci.com/}
\thanks{Declarations: Data sharing not applicable to this article as no
datasets were generated or analysed during the current study. The author has
no conicts of interest to declare that are relevant to the content of this
article.\\
I would like to thank all my family and friends for the support given to me
over the years: Elisa Cirri, Caterina Granucci, Delia Granucci, Irene
Granucci, Laura and Fiorenza Granucci, Massimo Masi and Monia Randolfi. Also
I would like to thank my professors Luigi Barletti, Giorgio Busoni, Elvira
Mascolo, Giorgio Talenti and Vincenzo Vepri }
\date{Sectember 27, 2021}
\dedicatory{Dedicated to the memory of Fiorella Pettini.}
\thanks{This paper is in final form and no version of it will be submitted
for publication elsewhere.}

\begin{abstract}
In this paper we study the everywhere H\"{o}der continuity of the minima of
the following class of vectorial integral funcionals 
\begin{equation*}
\int\limits_{\Omega }\sum\limits_{\alpha =1}^{m}\left\vert \nabla u^{\alpha
}\right\vert ^{p}+G\left( x,u,\left\vert \nabla u^{1}\right\vert
,...,\left\vert \nabla u^{m}\right\vert \right) \,dx
\end{equation*}%
with some general condidion on the density $G$. We make the following
assumptions about the function $G$. Let $\Omega $ be a bounded open subset
of $%
%TCIMACRO{\U{211d} }%
%BeginExpansion
\mathbb{R}
%EndExpansion
^{n}$ with $n\geq 2\ $and let $G:\Omega \times 
%TCIMACRO{\U{211d} }%
%BeginExpansion
\mathbb{R}
%EndExpansion
^{m}\times 
%TCIMACRO{\U{211d} }%
%BeginExpansion
\mathbb{R}
%EndExpansion
_{0,+}^{m}\rightarrow 
%TCIMACRO{\U{211d} }%
%BeginExpansion
\mathbb{R}
%EndExpansion
$ be a Caratheodory function, where $%
%TCIMACRO{\U{211d} }%
%BeginExpansion
\mathbb{R}
%EndExpansion
_{0,+}=\left[ 0,+\infty \right) $ $\ $and $%
%TCIMACRO{\U{211d} }%
%BeginExpansion
\mathbb{R}
%EndExpansion
_{0,+}^{m}=%
%TCIMACRO{\U{211d} }%
%BeginExpansion
\mathbb{R}
%EndExpansion
_{0,+}\times \cdots \times 
%TCIMACRO{\U{211d} }%
%BeginExpansion
\mathbb{R}
%EndExpansion
_{0,+}$ with $m\geq 1$; we make the following growth conditions on $G$:\
there exists a constant $L>1$ such that%
\begin{eqnarray*}
\sum\limits_{\alpha =1}^{m}\left\vert \xi ^{\alpha }\right\vert
^{q}-\sum\limits_{\alpha =1}^{m}\left\vert s^{\alpha }\right\vert
^{q}-a\left( x\right) &\leq &G\left( x,s^{1},...,s^{m},\left\vert \xi
^{1}\right\vert ,...,\left\vert \xi ^{m}\right\vert \right) \\
&\leq &L\left[ \sum\limits_{\alpha =1}^{m}\left\vert \xi ^{\alpha
}\right\vert ^{q}+\sum\limits_{\alpha =1}^{m}\left\vert s^{\alpha
}\right\vert ^{q}+a\left( x\right) \right]
\end{eqnarray*}%
for $\mathcal{L}^{n}$ a. e. $x\in \Omega $, for every $s^{\alpha }\in 
%TCIMACRO{\U{211d} }%
%BeginExpansion
\mathbb{R}
%EndExpansion
$ and for every $\xi ^{\alpha }\in 
%TCIMACRO{\U{211d} }%
%BeginExpansion
\mathbb{R}
%EndExpansion
$ with $\alpha =1,...,m$ and $m\geq 1$ and with $a\left( x\right) \in
L^{\sigma }\left( \Omega \right) $, $a(x)\geq 0$ for $\mathcal{L}^{n}$ a. e. 
$x\in \Omega $, $\sigma >\frac{n}{p}$, $1\leq q<\frac{p^{2}}{n}$ and $1<p<n$%
. \bigskip Assuming that the previous growth hypothesis holds, we prove the
following regularity result. if $u\in W^{1,p}\left( \Omega ,%
%TCIMACRO{\U{211d} }%
%BeginExpansion
\mathbb{R}
%EndExpansion
^{m}\right) $ is a local minimizer of the previous functional then $%
u^{\alpha }\in C_{loc}^{o,\beta _{0}}\left( \Omega \right) $ for every $%
\alpha =1,...,m$, with $\beta _{0}\in \left( 0,1\right) $. The regularity of
minimizers is obtained by proving that each component stays in a suitable De
Giorgi class and, from this, we conclude about the H\"{o}der continuity.
\end{abstract}

\maketitle

\newpage

\section{\protect\bigskip Introduction}

In this paper we study the everywhere regualiry of the minima of the
following class of vectorial integral functional of the calculus of variation%
\begin{equation}
\mathcal{F}\left( u,\Omega \right) =\int\limits_{\Omega }\sum\limits_{\alpha
=1}^{m}\left\vert \nabla u^{\alpha }\right\vert ^{p}+G\left( x,u,\left\vert
\nabla u^{1}\right\vert ,...,\left\vert \nabla u^{m}\right\vert \right) \,dx
\label{1.1}
\end{equation}

We make the following assumptions about the function $G$.

\begin{description}
\item[H.1] Let $\Omega $ be a bounded open subset of $%
%TCIMACRO{\U{211d} }%
%BeginExpansion
\mathbb{R}
%EndExpansion
^{n}$ with $n\geq 2\ $and let $G:\Omega \times 
%TCIMACRO{\U{211d} }%
%BeginExpansion
\mathbb{R}
%EndExpansion
^{m}\times 
%TCIMACRO{\U{211d} }%
%BeginExpansion
\mathbb{R}
%EndExpansion
_{0,+}^{m}\rightarrow 
%TCIMACRO{\U{211d} }%
%BeginExpansion
\mathbb{R}
%EndExpansion
$ be a Caratheodory function, where $%
%TCIMACRO{\U{211d} }%
%BeginExpansion
\mathbb{R}
%EndExpansion
_{0,+}=\left[ 0,+\infty \right) $ $\ $and $%
%TCIMACRO{\U{211d} }%
%BeginExpansion
\mathbb{R}
%EndExpansion
_{0,+}^{m}=%
%TCIMACRO{\U{211d} }%
%BeginExpansion
\mathbb{R}
%EndExpansion
_{0,+}\times \cdots \times 
%TCIMACRO{\U{211d} }%
%BeginExpansion
\mathbb{R}
%EndExpansion
_{0,+}$ with $m\geq 1$; we make the following growth conditions on $G$:\
there exists a constant $L>1$ such that%
\begin{equation*}
\sum\limits_{\alpha =1}^{m}\left\vert \xi ^{\alpha }\right\vert
^{q}-\sum\limits_{\alpha =1}^{m}\left\vert s^{\alpha }\right\vert
^{q}-a\left( x\right) \leq G\left( x,s^{1},...,s^{m},\left\vert \xi
^{1}\right\vert ,...,\left\vert \xi ^{m}\right\vert \right) \leq L\left[
\sum\limits_{\alpha =1}^{m}\left\vert \xi ^{\alpha }\right\vert
^{q}+\sum\limits_{\alpha =1}^{m}\left\vert s^{\alpha }\right\vert
^{q}+a\left( x\right) \right]
\end{equation*}%
for $\mathcal{L}^{n}$ a. e. $x\in \Omega $, for every $s^{\alpha }\in 
%TCIMACRO{\U{211d} }%
%BeginExpansion
\mathbb{R}
%EndExpansion
$ and for every $\xi ^{\alpha }\in 
%TCIMACRO{\U{211d} }%
%BeginExpansion
\mathbb{R}
%EndExpansion
$ with $\alpha =1,...,m$ and $m\geq 1$ and with $a\left( x\right) \in
L^{\sigma }\left( \Omega \right) $, $a(x)\geq 0$ for $\mathcal{L}^{n}$ a. e. 
$x\in \Omega $, $\sigma >\frac{n}{p}$, $1\leq q<\frac{p^{2}}{n}$ and $1<p<n$.
\end{description}

\bigskip Assuming that the previous growth hypothesis H.1 holds, we prove
the following regularity result.

\begin{theorem}
\label{th1} Let $\Omega $ be a bounded open subset of $%
%TCIMACRO{\U{211d} }%
%BeginExpansion
\mathbb{R}
%EndExpansion
^{n}$ with $n\geq 2$; if $u\in W^{1,p}\left( \Omega ,%
%TCIMACRO{\U{211d} }%
%BeginExpansion
\mathbb{R}
%EndExpansion
^{m}\right) $, with $m\geq 1$, is a local minimum of the functional \ref{1.1}
and $H.1$\ holds then $u^{\alpha }\in C_{loc}^{o,\beta _{0}}\left( \Omega
\right) $ for every $\alpha =1,...,m$, with $\beta _{0}\in \left( 0,1\right) 
$.
\end{theorem}

We find the previous result nontrivial and quite surprising for three
reasons; the first reason lies in the existence of numerous examples of
non-regular vector problems, refer to [13, 19, 23]; the second reason is due
to the existence, olso, of many regularity results for vectorial problems
but all these results have strong hypotheses on the density function, which
is either very regular or in some sense convex, refer to [1-11, 14-18, 20,
21, 25, 30, 35-39, 42-46], differently from the works cited we only assume
that the density $G$ is a Caratheodory function; finally the third reason
consists in the fact that the proof is extremely elementary, in the sense
that it uses known techniques introduced for the scalar case, referring to
[12, 22, 24, 28, 29, 31, 32, 33, 40, 41].

Recently, in [7-11, 25, 33] new classes of vectorial problems, with regular
weak solutions, have been introduced. In particular in [7], Cupini, Focardi,
Leonetti and Mascolo introduced the following class of vectorial functionals%
\begin{equation}
\int\limits_{\Omega }f\left( x,\nabla u\right) \,dx  \label{1.3}
\end{equation}%
where $\Omega \subset 
%TCIMACRO{\U{211d} }%
%BeginExpansion
\mathbb{R}
%EndExpansion
^{n},u:\Omega \rightarrow 
%TCIMACRO{\U{211d} }%
%BeginExpansion
\mathbb{R}
%EndExpansion
^{m},n>1,m\geq 1$ and%
\begin{equation*}
f\left( x,\nabla u\right) =\sum\limits_{\alpha =1}^{m}F_{\alpha }\left(
x,\nabla u^{\alpha }\right) +G\left( x,\nabla u\right)
\end{equation*}%
where $F_{\alpha }\times \ R^{n\times m}\rightarrow R$ is a Carath\'{e}odory
function satisfying the following standard growth condition%
\begin{equation*}
k_{1}\left\vert \xi ^{\alpha }\right\vert ^{p}-a\left( x\right) \leq
F_{\alpha }\left( x,\xi ^{\alpha }\right) \leq k_{2}\left\vert \xi ^{\alpha
}\right\vert ^{p}+a\left( x\right)
\end{equation*}%
for every $\xi ^{\alpha }\in 
%TCIMACRO{\U{211d} }%
%BeginExpansion
\mathbb{R}
%EndExpansion
^{n}$ and for almost every $x\in \Omega $ ,where $k_{1}$ and $k_{2}$ are two
real positive constants, $p>1$ and $a\in L_{loc}^{\sigma }\left( \Omega
\right) $ is a non negative function. In [7],Cupini, Focardi, Leonetti and
Mascolo analyze two different types of hypotheses on the $G$ function. They
started by assuming that $G:\Omega \times \ R^{n\times m}\rightarrow R$ is a
Carath\'{e}odory rank one convex function satisfying the following growth
condition%
\begin{equation*}
|G(x,\xi )|\leq k_{3}|\xi |^{q}+b(x)
\end{equation*}%
for every $\xi \in 
%TCIMACRO{\U{211d} }%
%BeginExpansion
\mathbb{R}
%EndExpansion
^{n\times m}$, for almost every $x\in \Omega $, here $k_{3}$ is a real
positive constant, $1\leq q<p$ and $b\in L_{loc}^{\sigma }\left( \Omega
\right) $ is a nonnegative function. Moreover Cupini, Focardi, Leonetti and
Mascolo in [7] study the case where $n\geq m\geq 3$, and $G:\Omega \times \ 
%TCIMACRO{\U{211d} }%
%BeginExpansion
\mathbb{R}
%EndExpansion
^{n\times m}\rightarrow R$ is a Carath\'{e}odory function defined as%
\begin{equation*}
G(x,\xi )=\sum\limits_{\alpha =1}^{m}G_{\alpha }\left( x,\left( adj_{m-1}\xi
\right) ^{\alpha }\right)
\end{equation*}%
here $G\alpha :\Omega \times \ 
%TCIMACRO{\U{211d} }%
%BeginExpansion
\mathbb{R}
%EndExpansion
^{\frac{m!}{n!\left( n-m\right) !}}\rightarrow 
%TCIMACRO{\U{211d} }%
%BeginExpansion
\mathbb{R}
%EndExpansion
$ is a Carath\'{e}odory convex function satisfying the following growth
conditions%
\begin{equation*}
0\leq G_{\alpha }\left( x,\left( adj_{m-1}\xi \right) ^{\alpha }\right) \leq
k_{4}|\left( adj_{m-1}\xi \right) ^{\alpha }|^{r}+b(x)
\end{equation*}%
for every $\xi \in 
%TCIMACRO{\U{211d} }%
%BeginExpansion
\mathbb{R}
%EndExpansion
^{n\times m}$, for almost every $x\in \Omega $, here $k_{3}$ is a real
positive constant, $1\leq r<p$ and $b\in L_{loc}^{\sigma }\left( \Omega
\right) $ is a non negative function. In both cases, by imposing appropriate
hypotheses on the parameters $q$ and $r$ , Cupini, Focardi, Leonetti and
Mascolo proved that the localminimizers of the vectorial functional (\ref%
{1.3}) are locally H\"{o}lder continuous functions. In [25] the author with
M. Randolfi proved a regularity result for the minima of vector functionals
with anisotropic growths of the following type%
\begin{equation}
\int\limits_{\Omega }\sum\limits_{\alpha =1}^{m}F_{\alpha }\left( x,\nabla
u^{\alpha }\right) +G\left( x,\nabla u\right) \,dx  \label{1.4}
\end{equation}%
with%
\begin{equation}
\sum\limits_{\alpha =1}^{m}\Phi _{i,\alpha }\left( \left\vert \xi
_{i}^{\alpha }\right\vert \right) \leq F_{\alpha }\left( x,\xi ^{\alpha
}\right) \leq L\left[ \bar{B}_{\alpha }^{\beta _{\alpha }}\left( \left\vert
\xi ^{\alpha }\right\vert \right) +a\left( x\right) \right]  \label{1.5}
\end{equation}%
where $\Phi _{i,\alpha }$ are N functions belonging to the class $\triangle
_{2}^{m_{\alpha }}\cap \nabla _{2}^{r_{\alpha }}$, $\bar{B}_{\alpha }$\ is
the Sobolev function associated with $\Phi _{i,\alpha }$'s, $\beta _{\alpha
}\in \left( 0,1\right] $ and $a\in L_{loc}^{\sigma }\left( \Omega \right) $
is a non negative function; oppure with%
\begin{equation}
\sum\limits_{\alpha =1}^{m}\Phi _{i,\alpha }\left( \left\vert \xi
_{i}^{\alpha }\right\vert \right) -a\left( x\right) \leq F_{\alpha }\left(
x,\xi ^{\alpha }\right) \leq L_{1}\left[ \sum\limits_{\alpha =1}^{m}\Phi
_{i,\alpha }\left( \left\vert \xi _{i}^{\alpha }\right\vert \right) +a\left(
x\right) \right]  \label{1.6}
\end{equation}%
where $\Phi _{i,\alpha }$ are N-functions belonging to the class $\triangle
_{2}^{m_{\alpha }}\cap \nabla _{2}^{r_{\alpha }}$ and $a\in L_{loc}^{\sigma
}\left( \Omega \right) $ is a non negative function, moreover, appropriate
hypotheses are made on the density $G$, for more details we refer to [25].
In particular, using the techniques presented in [25, 26, 28, 29], the
author with M. Randolfi have shown that the minima of the functional (\ref%
{1.4}) are locally bounded functions in the case (\ref{1.5}) and locally
Holder continuous in the case (\ref{1.6}), we refer to [25] for more
details. In recent years a large number of articles have been written
dealing with the regularity of weak solutions of vactorial problems under
standard growths, general growths and anisotropic growths, refer to [1-11,
14-18, 20, 21, 25, 30 , 35-39, 42 -46]. Our result falls within the research
scope proposed in [7, 8, 25] but differs from the latter for the lack of
regularity and convexity hypotheses, which are replaced by a particular
density structure $G$.

For completeness we remember that if $\Omega $ is a open subset of $%
%TCIMACRO{\U{211d} }%
%BeginExpansion
\mathbb{R}
%EndExpansion
^{n}$ and $u$\ is a Lebesgue measurable function then $L^{p}\left( \Omega
\right) $ is the set of the class of the Lebesgue measurable function such
that $\int\limits_{\Omega }\left\vert u\right\vert ^{p}\,dx<+\infty $ and $%
W^{1,p}\left( \Omega \right) $\ is the set of the function $u\in L^{p}\left(
\Omega \right) $ such that its waek derivate $\partial _{i}u\in L^{p}\left(
\Omega \right) $. The spaces $L^{p}\left( \Omega \right) $ and $%
W^{1,p}\left( \Omega \right) $ are Banach spaces with the respective norms%
\begin{equation*}
\left\Vert u\right\Vert _{L^{p}\left( \Omega \right) }=\left(
\int\limits_{\Omega }\left\vert u\right\vert ^{p}\,dx\right) ^{\frac{1}{p}}
\end{equation*}%
and%
\begin{equation*}
\left\Vert u\right\Vert _{W^{1,p}\left( \Omega \right) }=\left\Vert
u\right\Vert _{L^{p}\left( \Omega \right) }+\sum\limits_{i=1}^{n}\left\Vert
\partial _{i}u\right\Vert _{L^{p}\left( \Omega \right) }
\end{equation*}%
We say that the function $u:\Omega \subset 
%TCIMACRO{\U{211d} }%
%BeginExpansion
\mathbb{R}
%EndExpansion
^{n}\rightarrow 
%TCIMACRO{\U{211d} }%
%BeginExpansion
\mathbb{R}
%EndExpansion
^{m}$ belong in $W^{1,p}\left( \Omega ,%
%TCIMACRO{\U{211d} }%
%BeginExpansion
\mathbb{R}
%EndExpansion
^{m}\right) $ if $u^{\alpha }\in W^{1,p}\left( \Omega \right) $ for every $%
\alpha =1,...,m$, where $u^{\alpha }$ is the $\alpha $ component of the
vector-valued function $u$; we end by remembering that $W^{1,p}\left( \Omega
,%
%TCIMACRO{\U{211d} }%
%BeginExpansion
\mathbb{R}
%EndExpansion
^{m}\right) $ is a Banach space with the norm%
\begin{equation*}
\left\Vert u\right\Vert _{W^{1,p}\left( \Omega ,%
%TCIMACRO{\U{211d} }%
%BeginExpansion
\mathbb{R}
%EndExpansion
^{n}\right) }=\sum\limits_{\alpha =1}^{m}\left\Vert u^{\alpha }\right\Vert
_{W^{1,p}\left( \Omega \right) }
\end{equation*}

\section{Caccioppoli inequalities}

To prove our regularity theorem, Theorem \ref{th1}, we need to introduce
some fundamental regularity results introduced in the 70s that generalize
the famous Theorem of De Giorgi-Nash-Moser, refer to [12, 40, 41], in
particular we will refer to the results of [22] as presented in [22] and
[24].

\begin{definition}
Let $\Omega \subset 
%TCIMACRO{\U{211d} }%
%BeginExpansion
\mathbb{R}
%EndExpansion
^{n}$ be a bounded open set with $n\geq 2$ and $v:\Omega \rightarrow 
%TCIMACRO{\U{211d} }%
%BeginExpansion
\mathbb{R}
%EndExpansion
$, we say that $v\in W_{loc}^{1,p}\left( \Omega \right) $ belong to the De
Giorgi class $DG^{+}\left( \Omega ,p,\lambda ,\lambda _{\ast },\chi
,\varepsilon ,R_{0},k_{0}\right) $ with $p>1$, $\lambda >0$, $\lambda _{\ast
}>0$, $\chi >0$, $\varepsilon >0$, $R_{0}>0$ and $k_{0}\geq 0$ if%
\begin{equation}
\int\limits_{A_{k,\varrho }}\left\vert \nabla v\right\vert ^{p}\,dx\leq 
\frac{\lambda }{\left( R-\varrho \right) ^{p}}\int\limits_{A_{k,R}}\left(
v-k\right) ^{p}\,dx+\lambda _{\ast }\left( \chi ^{p}+k^{p}R^{-n\varepsilon
}\right) \left\vert A_{k,R}\right\vert ^{1-\frac{p}{n}+\varepsilon }
\end{equation}%
for all $k\geq k_{0}\geq 0$ and for all pair of balls $B_{\varrho }\left(
x_{0}\right) \subset B_{R}\left( x_{0}\right) \subset \subset \Omega $ with $%
0<\varrho <R<R_{0}$ and $A_{k,s}=B_{s}\left( x_{0}\right) \cap \left\{
v>k\right\} $ with $s>0$.
\end{definition}

\begin{definition}
Let $\Omega \subset 
%TCIMACRO{\U{211d} }%
%BeginExpansion
\mathbb{R}
%EndExpansion
^{n}$ be a bounded open set with $n\geq 2$ and $v:\Omega \rightarrow 
%TCIMACRO{\U{211d} }%
%BeginExpansion
\mathbb{R}
%EndExpansion
$, we say that $v\in W_{loc}^{1,p}\left( \Omega \right) $ belong to the De
Giorgi class $DG^{-}\left( \Omega ,p,\lambda ,\lambda _{\ast },\chi
,\varepsilon ,R_{0},k_{0}\right) $ with $p>1$, $\lambda >0$, $\lambda _{\ast
}>0$, $\chi >0$ and $k_{0}\geq 0$ if%
\begin{equation}
\int\limits_{B_{k,\varrho }}\left\vert \nabla v\right\vert ^{p}\,dx\leq 
\frac{\lambda }{\left( R-\varrho \right) ^{p}}\int\limits_{B_{k,R}}\left(
k-v\right) ^{p}\,dx+\lambda _{\ast }\left( \chi ^{p}+\left\vert k\right\vert
^{p}R^{-n\varepsilon }\right) \left\vert B_{k,R}\right\vert ^{1-\frac{p}{n}%
+\varepsilon }
\end{equation}%
for all $k\leq -k_{0}\leq 0$ and for all pair of balls $B_{\varrho }\left(
x_{0}\right) \subset B_{R}\left( x_{0}\right) \subset \subset \Omega $ with $%
0<\varrho <R<R_{0}$ and $B_{k,s}=B_{s}\left( x_{0}\right) \cap \left\{
v<k\right\} $ with $s>0$.
\end{definition}

\begin{definition}
We set $DG\left( \Omega ,p,\lambda ,\lambda _{\ast },\chi ,\varepsilon
,R_{0},k_{0}\right) =DG^{+}\left( \Omega ,p,\lambda ,\lambda _{\ast },\chi
,\varepsilon ,R_{0},k_{0}\right) \cap DG^{-}\left( \Omega ,p,\lambda
,\lambda _{\ast },\chi ,\varepsilon ,R_{0},k_{0}\right) $. Moreover $u\in
DG\left( \Omega ,p,\lambda ,\lambda _{\ast },\chi ,\varepsilon ,R_{0}\right) 
$ if $u\in DG\left( \Omega ,p,\lambda ,\lambda _{\ast },\chi ,\varepsilon
,R_{0},k_{0}\right) $ for every $k_{0}\in 
%TCIMACRO{\U{211d} }%
%BeginExpansion
\mathbb{R}
%EndExpansion
$. \ 
\end{definition}

\begin{theorem}
\label{th2} Let $v\in DG\left( \Omega ,p,\lambda ,\lambda _{\ast },\chi
,\varepsilon ,R_{0}\right) $ and $\tau \in (0,1)$, then there exists a
constant $C>1$ depending only upon the data and indipendent fo $v$ and $%
x_{0}\in \Omega $ such that for every pair of balls $B_{\tau \varrho }\left(
x_{0}\right) \subset B_{\varrho }\left( x_{0}\right) \subset \subset \Omega $
with $0<\varrho <R_{0}$ 
\begin{equation}
\left\Vert v\right\Vert _{L^{\infty }\left( B_{\tau \varrho }\left(
x_{0}\right) \right) }\leq \max \left\{ \lambda _{\ast }\varrho ^{\frac{%
n\varepsilon }{p}};\frac{C}{\left( 1-\tau \right) ^{\frac{n}{p}}}\left[ 
\frac{1}{\left\vert B_{\varrho }\left( x_{0}\right) \right\vert }%
\int\limits_{B_{\varrho }\left( x_{0}\right) }\left\vert v\right\vert
^{p}\,dx\right] ^{\frac{1}{p}}\right\}
\end{equation}
\end{theorem}

\begin{proof}
We refer to Theorem 7.2 and Theorem 7.4 of [24]
\end{proof}

For more details on De Giorgi's classes and for the proof of the Theorem \ref%
{th2} we refer to [12, 22, 24, 28, 29, 31, 32, 33, 35].

The fundamental result of this section is the following Theorem.

\begin{theorem}
\label{th3} Let $\Omega \subset 
%TCIMACRO{\U{211d} }%
%BeginExpansion
\mathbb{R}
%EndExpansion
^{n}$ be a bounded open set with $n\geq 2$; if $u\in W^{1,p}\left( \Omega ,%
%TCIMACRO{\U{211d} }%
%BeginExpansion
\mathbb{R}
%EndExpansion
^{m}\right) $, with $m\geq 1$, is a minimum of the functional 1.1 then $%
u^{\alpha }\in DG\left( \Omega ,p,\lambda ,\lambda _{\ast },\chi
,\varepsilon ,R_{0}\right) $ for all $\alpha =1,...,m$.
\end{theorem}

\section{Lemmata}

To complete the proof of Theorem \ref{th1} it remains to prove Theorem \ref%
{th3}. Before giving the proof of Theorem \ref{th3}, for completeness we
introduce a list of results that we will use during the proof.

\begin{lemma}[Young Inequality]
Let $\varepsilon >0$, $a,b>0$ and $1<p,q<+\infty $ with $\frac{1}{p}+\frac{1%
}{q}=1$\ then it follows%
\begin{equation}  \label{4.1}
ab\leq \varepsilon \frac{a^{p}}{p}+\frac{b^{q}}{\varepsilon ^{\frac{q}{p}}q}
\end{equation}
\end{lemma}

\begin{lemma}[H\"{o}lder Inequality]
Assume $1\leq p,q\leq +\infty $ with $\frac{1}{p}+\frac{1}{q}=1$\ then if $%
u\in L^{p}\left( \Omega \right) $\ and $v\in L^{p}\left( \Omega \right) $\
it follows%
\begin{equation}  \label{4.2}
\int\limits_{\Omega }\left\vert uv\right\vert \,dx\leq \left(
\int\limits_{\Omega }\left\vert u\right\vert ^{p}\,dx\right) ^{\frac{1}{p}%
}\left( \int\limits_{\Omega }\left\vert v\right\vert ^{q}\,dx\right) ^{\frac{%
1}{q}}
\end{equation}
\end{lemma}

\begin{lemma}
\label{lemma3} Let $Z\left( t\right) $ be a nonnegative and bounded function
on the set $\left[ \varrho ,R\right] $; if for every $\varrho \leq t<s\leq R$
we get%
\begin{equation*}
Z\left( t\right) \leq \theta Z\left( s\right) +\frac{A}{\left( s-t\right)
^{\lambda }}+\frac{B}{\left( s-t\right) ^{\mu }}+C
\end{equation*}%
where $A,B,C\geq 0$, $\lambda >\mu >0$ and $0\leq \theta <1$ then it follows%
\begin{equation*}
Z\left( \varrho \right) \leq C\left( \theta ,\lambda \right) \left( \frac{A}{%
\left( R-\varrho \right) ^{\lambda }}+\frac{B}{\left( R-\varrho \right)
^{\mu }}+C\right)
\end{equation*}%
where $C\left( \theta ,\lambda \right) >0$ is a real constant depending only
on $\theta $ and $\lambda $.
\end{lemma}

\begin{theorem}[Sobolev Inequlity]
Let $\Omega $ be a open subset of $%
%TCIMACRO{\U{211d} }%
%BeginExpansion
\mathbb{R}
%EndExpansion
^{N}$ if $u\in W_{0}^{1,p}\left( \Omega \right) $ with $1\leq p<N$ there
exists a real positive constant $C_{SN}$, depending only on $p$ and $N$,
such that%
\begin{equation}
\left\Vert u\right\Vert _{L^{p^{\ast }}\left( \Omega \right) }\leq
C_{SN}\left\Vert \nabla u\right\Vert _{L^{p}\left( \Omega \right) }
\end{equation}%
where $p^{\ast }=\frac{Np}{N-p}$.
\end{theorem}

\begin{theorem}
(Rellich-Sobolev Embendding Theorem) Let $\Omega $ be a open bounded subset
of $%
%TCIMACRO{\U{211d} }%
%BeginExpansion
\mathbb{R}
%EndExpansion
^{N}$ with lipschitz bondary then if $u\in W^{1,p}\left( \Omega \right) $
with $1\leq p<N$ there exists a real positive constant $C_{ES}$, depending
only on $p$ and $N$, such that%
\begin{equation}
\left\Vert u\right\Vert _{L^{p^{\ast }}\left( \Omega \right) }\leq
C_{ES}\left\Vert u\right\Vert _{W^{1,p}\left( \Omega \right) }
\end{equation}%
where $p^{\ast }=\frac{Np}{N-p}$.
\end{theorem}

For more details on Sobolev's spaces, refer to [24].

\section{Proof of Teorem \protect\ref{th3}}

In this section we will give the proof of Theorem \ref{th3}. \ Let $\Sigma
\subset \subset \Omega $ be a compact subset of $\Omega $, let $R_{0}$ be a
positive real number, let us consider $x_{0}\in \Sigma $, $0<\varrho \leq
t<s\leq R<\min \left( R_{0},1,\frac{dist\left( x_{0},\partial \Sigma \right) 
}{2}\right) $, $k\in 
%TCIMACRO{\U{211d} }%
%BeginExpansion
\mathbb{R}
%EndExpansion
$\ and $\eta \in C_{c}^{\infty }\left( B_{s}\left( x_{0}\right) \right) $
such that $0\leq \eta \leq 1$ on $B_{s}\left( x_{0}\right) $, $\eta =1$ in $%
B_{t}\left( x_{0}\right) $ and $\left\vert \nabla \eta \right\vert \leq 
\frac{2}{s-t}$ on $B_{s}\left( x_{0}\right) $, moreover we we can observe
that supp$\left( \nabla \eta \right) \subset B_{s}\left( x_{0}\right)
\backslash B_{t}\left( x_{0}\right) $. If $u_{\varepsilon }$ is a local
sub-minimum of $J_{\varepsilon }\left( u,\Omega \right) $ we can define $%
\varphi =-\eta ^{p}w_{+}$ where%
\begin{equation*}
w_{+}=\left( 
\begin{tabular}{l}
$\left( u^{1}-k\right) _{+}$ \\ 
$0$ \\ 
$\vdots $ \\ 
$0$%
\end{tabular}%
\right)
\end{equation*}%
with $\left( u^{1}-k\right) _{+}=\max \left\{ u^{1}-k,0\right\} $ and,
moreover, we get%
\begin{equation*}
\mathcal{F}\left( u,B_{s}\left( x_{0}\right) \right) \leq \mathcal{F}\left(
u+\varphi ,B_{s}\left( x_{0}\right) \right)
\end{equation*}%
\ The previous variational inequality can be written like this%
\begin{equation*}
\begin{tabular}{l}
$\int\limits_{A_{k,s}^{1}}\left\vert \nabla u^{1}\right\vert
^{p}\,dx+\int\limits_{A_{k,s}^{1}}\sum\limits_{\alpha =2}^{m}\left\vert
\nabla u^{\alpha }\right\vert ^{p}\,dx+\int\limits_{B_{s}\left( x_{0}\right)
\backslash A_{k,s}^{1}}\sum\limits_{\alpha =1}^{m}\left\vert \nabla
u^{\alpha }\right\vert ^{p}\,dx$ \\ 
$+\int\limits_{A_{k,s}^{1}}G\left( x,u^{1},...,u^{m},\left\vert \nabla
u^{1}\right\vert ,...,\left\vert \nabla u^{m}\right\vert \right)
\,dx+\int\limits_{B_{s}\left( x_{0}\right) \backslash A_{k,s}^{1}}G\left(
x,u,\left\vert \nabla u^{1}\right\vert ,...,\left\vert \nabla
u^{m}\right\vert \right) \,dx$ \\ 
$\leq \int\limits_{A_{k,s}^{1}}\left( 1-\eta ^{p}\right) ^{p}\left\vert
\nabla u^{1}\right\vert ^{p}\,dx+p^{p}\int\limits_{A_{k,s}^{1}\backslash
A_{k,t}^{1}}\eta ^{\left( p-1\right) p}\left\vert \nabla \eta \right\vert
^{p}\left( u^{1}-k\right) ^{p}\,dx$ \\ 
$+\int\limits_{A_{k,s}^{1}}\sum\limits_{\alpha =2}^{m}\left\vert \nabla
u^{\alpha }\right\vert ^{p}\,dx+\int\limits_{B_{s}\left( x_{0}\right)
\backslash A_{k,s}^{1}}\sum\limits_{\alpha =1}^{m}\left\vert \nabla
u^{\alpha }\right\vert ^{p}\,dx$ \\ 
$+\int\limits_{A_{k,s}^{1}}G\left( x,\left( 1-\eta ^{p}\right) \left(
u^{1}-k\right) +k,...,u^{m},\left\vert \left( 1-\eta ^{p}\right) \nabla
u^{1}+p\eta ^{p-1}\nabla \eta \left( u^{1}-k\right) \right\vert
,...,\left\vert \nabla u^{m}\right\vert \right) \,dx$ \\ 
$+\int\limits_{B_{s}\left( x_{0}\right) \backslash A_{k,s}^{1}}G\left(
x,u,\left\vert \nabla u^{1}\right\vert ,...,\left\vert \nabla
u^{m}\right\vert \right) \,dx$%
\end{tabular}%
\end{equation*}%
where $A_{k,\varrho }^{1}=\left\{ u^{1}>k\right\} \cap B_{\varrho }\left(
x_{0}\right) $\ and with simple algebraic calculations then we get%
\begin{equation*}
\begin{tabular}{l}
$\int\limits_{A_{k,s}^{1}}\left\vert \nabla u^{1}\right\vert
^{p}\,dx+\int\limits_{A_{k,s}^{1}}G\left( x,u^{1},...,u^{m},\left\vert
\nabla u^{1}\right\vert ,...,\left\vert \nabla u^{m}\right\vert \right) \,dx$
\\ 
$\leq \int\limits_{A_{k,s}^{1}}\left( 1-\eta ^{p}\right) \left\vert \nabla
u^{1}\right\vert \,dx+p^{p}\int\limits_{A_{k,s}^{1}\backslash
A_{k,t}^{1}}\eta ^{p-1}\left\vert \nabla \eta \right\vert ^{p}\left(
u^{1}-k\right) ^{p}\,dx$ \\ 
$+\int\limits_{A_{k,s}^{1}}G\left( x,\left( 1-\eta ^{p}\right) \left(
u^{1}-k\right) +k,...,u^{m},\left\vert \left( 1-\eta ^{p}\right) \nabla
u^{1}+p\eta ^{p-1}\nabla \eta \left( u^{1}-k\right) \right\vert
,...,\left\vert \nabla u^{m}\right\vert \right) \,dx$%
\end{tabular}%
\end{equation*}%
Recalling that from the hypothesis H.1 we have%
\begin{equation*}
\sum\limits_{\alpha =1}^{m}\left\vert \nabla u^{\alpha }\right\vert
^{q}-\sum\limits_{\alpha =1}^{m}\left\vert u^{\alpha }\right\vert
^{q}-a\left( x\right) \leq G\left( x,u^{1},...,u^{m},\left\vert \nabla
u^{1}\right\vert ,...,\left\vert \nabla u^{m}\right\vert \right)
\end{equation*}%
and%
\begin{equation*}
\begin{tabular}{l}
$G\left( x,\left( 1-\eta ^{p}\right) \left( u^{1}-k\right)
+k,...,u^{m},\left\vert \left( 1-\eta ^{p}\right) \nabla u^{1}+p\eta
^{p-1}\nabla \eta \left( u^{1}-k\right) \right\vert ,...,\left\vert \nabla
u^{m}\right\vert \right) $ \\ 
$\leq L\left\vert \left( 1-\eta ^{p}\right) \nabla u^{1}+p\eta ^{p-1}\nabla
\eta \left( u^{1}-k\right) \right\vert ^{q}+L\sum\limits_{\alpha
=2}^{m}\left\vert \nabla u^{\alpha }\right\vert ^{q}$ \\ 
$+L\left\vert \left( 1-\eta ^{p}\right) \left( u^{1}-k\right) +k\right\vert
^{q}+L\sum\limits_{\alpha =2}^{m}\left\vert u^{\alpha }\right\vert
^{q}+La\left( x\right) $%
\end{tabular}%
\end{equation*}%
then we get%
\begin{equation*}
\begin{tabular}{l}
$\int\limits_{A_{k,s}^{1}}\left\vert \nabla u^{1}\right\vert
^{p}\,dx+\int\limits_{A_{k,s}^{1}}\sum\limits_{\alpha =1}^{m}\left\vert
\nabla u^{\alpha }\right\vert ^{q}-\sum\limits_{\alpha =1}^{m}\left\vert
u^{\alpha }\right\vert ^{q}-a\left( x\right) \,dx$ \\ 
$\leq \int\limits_{A_{k,s}^{1}}\left( 1-\eta ^{p}\right) \left\vert \nabla
u^{1}\right\vert ^{p}\,dx+p^{p}\int\limits_{A_{k,s}^{1}\backslash
A_{k,t}^{1}}\eta ^{p-1}\left\vert \nabla \eta \right\vert ^{p}\left(
u^{1}-k\right) ^{p}\,dx$ \\ 
$+\int\limits_{A_{k,s}^{1}}L\left\vert \left( 1-\eta ^{p}\right) \nabla
u^{1}+p\eta ^{p-1}\nabla \eta \left( u^{1}-k\right) \right\vert
^{q}+L\sum\limits_{\alpha =2}^{m}\left\vert \nabla u^{\alpha }\right\vert
^{q}\,dx$ \\ 
$+\int\limits_{A_{k,s}^{1}}L\left\vert \left( 1-\eta ^{p}\right) \left(
u^{1}-k\right) +k\right\vert ^{q}+L\sum\limits_{\alpha =2}^{m}\left\vert
u^{\alpha }\right\vert ^{q}+La\left( x\right) \,dx$%
\end{tabular}%
\end{equation*}%
\bigskip \bigskip From which, with simple weighting we have%
\begin{equation*}
\begin{tabular}{l}
$\int\limits_{A_{k,s}^{1}}\left\vert \nabla u^{1}\right\vert
^{p}\,dx+\int\limits_{A_{k,s}^{1}}\sum\limits_{\alpha =1}^{m}\left\vert
\nabla u^{\alpha }\right\vert ^{q}\,dx$ \\ 
$\leq \int\limits_{A_{k,s}^{1}}\left( 1-\eta ^{p}\right) \left\vert \nabla
u^{1}\right\vert ^{p}\,dx+p^{p}\int\limits_{A_{k,s}^{1}\backslash
A_{k,t}^{1}}\eta ^{p-1}\left\vert \nabla \eta \right\vert ^{p}\left(
u^{1}-k\right) ^{p}\,dx$ \\ 
$+\int\limits_{A_{k,s}^{1}}\sum\limits_{\alpha =1}^{m}\left\vert u^{\alpha
}\right\vert ^{q}+a\left( x\right) \,dx$ \\ 
$+\int\limits_{A_{k,s}^{1}}L\left( 1-\eta ^{p}\right) ^{q}\left\vert \nabla
u^{1}\right\vert ^{q}+p^{q}\eta ^{\left( p-1\right) q}\left\vert \nabla \eta
\right\vert ^{q}\left( u^{1}-k\right) ^{q}+L\sum\limits_{\alpha
=2}^{m}\left\vert \nabla u^{\alpha }\right\vert ^{q}\,dx$ \\ 
$+\int\limits_{A_{k,s}^{1}}L\left\vert \left( 1-\eta ^{p}\right) \left(
u^{1}-k\right) +k\right\vert ^{q}+L\sum\limits_{\alpha =2}^{m}\left\vert
u^{\alpha }\right\vert ^{q}+La\left( x\right) \,dx$%
\end{tabular}%
\end{equation*}%
Since $1-\eta ^{p}=0$ on $A_{k,t}^{1}$, it follows that%
\begin{equation}
\begin{tabular}{l}
$\ \int\limits_{A_{k,s}^{1}}\left\vert \nabla u^{1}\right\vert
^{p}\,+\left\vert \nabla u^{1}\right\vert ^{q}\,dx$ \\ 
$\leq L\int\limits_{A_{k,s}^{1}\backslash A_{k,t}^{1}}\left\vert \nabla
u^{1}\right\vert ^{p}+\left\vert \nabla u^{1}\right\vert
^{q}\,dx+p^{p}\int\limits_{A_{k,s}^{1}\backslash A_{k,t}^{1}}\eta
^{p-1}\left\vert \nabla \eta \right\vert ^{p}\left( u^{1}-k\right) ^{p}\,dx$
\\ 
$+\int\limits_{A_{k,s}^{1}}\sum\limits_{\alpha =1}^{m}\left\vert u^{\alpha
}\right\vert ^{q}\,dx+\int\limits_{A_{k,s}^{1}}L\left\vert \left( 1-\eta
^{p}\right) \left( u^{1}-k\right) +k\right\vert ^{q}+L\sum\limits_{\alpha
=2}^{m}\left\vert u^{\alpha }\right\vert ^{q}+\left( L+1\right) a\left(
x\right) \,dx$ \\ 
$+\int\limits_{A_{k,s}^{1}}p^{q}\eta ^{\left( p-1\right) q}\left\vert \nabla
\eta \right\vert ^{q}\left( u^{1}-k\right) ^{q}+L\sum\limits_{\alpha
=2}^{m}\left\vert \nabla u^{\alpha }\right\vert ^{q}\,dx$%
\end{tabular}
\label{5.1}
\end{equation}%
We observe that we can estimate the term 
\begin{equation*}
\int\limits_{A_{k,s}^{1}}\sum\limits_{\alpha =1}^{m}\left\vert u^{\alpha
}\right\vert ^{q}\,+L\sum\limits_{\alpha =2}^{m}\left\vert u^{\alpha
}\right\vert ^{q}+\left( L+1\right) a\left( x\right) \,dx
\end{equation*}%
using the H\"{o}lder inequality \ref{4.2} obtaining%
\begin{equation}
\begin{tabular}{l}
$\int\limits_{A_{k,s}^{1}}\sum\limits_{\alpha =1}^{m}\left\vert u^{\alpha
}\right\vert ^{q}\,+L\sum\limits_{\alpha =2}^{m}\left\vert u^{\alpha
}\right\vert ^{q}+\left( L+1\right) a\left( x\right) \,dx$ \\ 
$\leq m\left( L+1\right) \left[ \mathcal{L}^{n}\left( A_{k,s}^{1}\right) %
\right] ^{1-\frac{q}{p}}\left[ \int\limits_{A_{k,s}^{1}}\left\vert
u\right\vert ^{p}\,dx\right] ^{\frac{q}{p}}$ \\ 
$+\left( L+1\right) \left[ \mathcal{L}^{n}\left( A_{k,s}^{1}\right) \right]
^{1-\frac{1}{\sigma }}\left[ \int\limits_{A_{k,s}^{1}}a^{\sigma }\,dx\right]
^{\frac{1}{\sigma }}$%
\end{tabular}
\label{5.2}
\end{equation}%
Similarly we can also estimate the term%
\begin{equation*}
\int\limits_{A_{k,s}^{1}}L\sum\limits_{\alpha =2}^{m}\left\vert \nabla
u^{\alpha }\right\vert ^{q}\,dx
\end{equation*}%
using the H\"{o}lder inequality \ref{4.2} obtaining%
\begin{equation}
\begin{tabular}{l}
$\int\limits_{A_{k,s}^{1}}L\sum\limits_{\alpha =2}^{m}\left\vert \nabla
u^{\alpha }\right\vert ^{q}\,dx$ \\ 
$\leq mL\left[ \mathcal{L}^{n}\left( A_{k,s}^{1}\right) \right] ^{1-\frac{q}{%
p}}\left[ \int\limits_{A_{k,s}^{1}}\left\vert \nabla u\right\vert ^{p}\,dx%
\right] ^{\frac{q}{p}}$%
\end{tabular}
\label{5.3}
\end{equation}%
Using (\ref{5.1}), (\ref{5.2}) and (\ref{5.3})\ we get%
\begin{equation}
\begin{tabular}{l}
$\ \int\limits_{A_{k,s}^{1}}\left\vert \nabla u^{1}\right\vert
^{p}\,+\left\vert \nabla u^{1}\right\vert ^{q}\,dx$ \\ 
$\leq L\int\limits_{A_{k,s}^{1}\backslash A_{k,t}^{1}}\left\vert \nabla
u^{1}\right\vert ^{p}+\left\vert \nabla u^{1}\right\vert
^{q}\,dx+p^{p}\int\limits_{A_{k,s}^{1}\backslash A_{k,t}^{1}}\eta
^{p-1}\left\vert \nabla \eta \right\vert ^{p}\left( u^{1}-k\right) ^{p}\,dx$
\\ 
$+\int\limits_{A_{k,s}^{1}}p^{q}\eta ^{\left( p-1\right) q}\left\vert \nabla
\eta \right\vert ^{q}\left( u^{1}-k\right)
^{q}\,dx+\int\limits_{A_{k,s}^{1}}L\left\vert \left( 1-\eta ^{p}\right)
\left( u^{1}-k\right) +k\right\vert ^{q}\,dx$ \\ 
$+2m\left( L+1\right) \left[ \mathcal{L}^{n}\left( A_{k,s}^{1}\right) \right]
^{1-\frac{q}{p}}\left\Vert u\right\Vert _{W^{1,p}\left( A_{k,s}^{1}\right)
}^{q}+\left( L+1\right) \left[ \mathcal{L}^{n}\left( A_{k,s}^{1}\right) %
\right] ^{1-\frac{1}{\sigma }}\left\Vert a\right\Vert _{L^{\sigma }\left(
A_{k,s}^{1}\right) }$%
\end{tabular}
\label{5.4}
\end{equation}%
Using the Young Inequality \ref{4.1} we have%
\begin{equation}
\begin{tabular}{l}
$\ \int\limits_{A_{k,s}^{1}}p^{q}\eta ^{\left( p-1\right) q}\left\vert
\nabla \eta \right\vert ^{q}\left( u^{1}-k\right) ^{q}\,dx$ \\ 
$\leq p^{q}\int\limits_{A_{k,s}^{1}}1+\left\vert \nabla \eta \right\vert
^{p}\left( u^{1}-k\right) ^{p}\,dx$ \\ 
$\leq p^{q}\mathcal{L}^{n}\left( A_{k,s}^{1}\right)
+p^{q}\int\limits_{A_{k,s}^{1}}\left\vert \nabla \eta \right\vert ^{p}\left(
u^{1}-k\right) ^{p}\,dx$%
\end{tabular}
\label{5.5}
\end{equation}%
then using (\ref{5.4}) and (\ref{5.5}) it follows%
\begin{equation*}
\begin{tabular}{l}
$\ \int\limits_{A_{k,s}^{1}}\left\vert \nabla u^{1}\right\vert
^{p}\,+\left\vert \nabla u^{1}\right\vert ^{q}\,dx$ \\ 
$\leq L\int\limits_{A_{k,s}^{1}\backslash A_{k,t}^{1}}\left\vert \nabla
u^{1}\right\vert ^{p}+\left\vert \nabla u^{1}\right\vert ^{q}\,dx+\left(
p^{p}+p^{q}\right) \int\limits_{A_{k,s}^{1}\backslash A_{k,t}^{1}}\eta
^{p-1}\left\vert \nabla \eta \right\vert ^{p}\left( u^{1}-k\right) ^{p}\,dx$
\\ 
$+p^{q}\mathcal{L}^{n}\left( A_{k,s}^{1}\right)
+\int\limits_{A_{k,s}^{1}}L\left\vert \left( 1-\eta ^{p}\right) \left(
u^{1}-k\right) +k\right\vert ^{q}\,dx$ \\ 
$+2m\left( L+1\right) \left[ \mathcal{L}^{n}\left( A_{k,s}^{1}\right) \right]
^{1-\frac{q}{p}}\left\Vert u\right\Vert _{W^{1,p}\left( A_{k,s}^{1}\right)
}^{q}+\left( L+1\right) \left[ \mathcal{L}^{n}\left( A_{k,s}^{1}\right) %
\right] ^{1-\frac{1}{\sigma }}\left\Vert a\right\Vert _{L^{\sigma }\left(
A_{k,s}^{1}\right) }$%
\end{tabular}%
\end{equation*}%
and%
\begin{equation*}
\begin{tabular}{l}
$\ \int\limits_{A_{k,s}^{1}}\left\vert \nabla u^{1}\right\vert
^{p}\,+\left\vert \nabla u^{1}\right\vert ^{q}\,dx$ \\ 
$\leq L\int\limits_{A_{k,s}^{1}\backslash A_{k,t}^{1}}\left\vert \nabla
u^{1}\right\vert ^{p}+\left\vert \nabla u^{1}\right\vert ^{q}\,dx+\left(
p^{p}+p^{q}\right) \int\limits_{A_{k,s}^{1}\backslash A_{k,t}^{1}}\eta
^{p-1}\left\vert \nabla \eta \right\vert ^{p}\left( u^{1}-k\right) ^{p}\,dx+$
\\ 
$+\left( p^{q}+L\right) \mathcal{L}^{n}\left( A_{k,s}^{1}\right)
+2^{p-1}L\int\limits_{A_{k,s}^{1}\backslash A_{k,t}^{1}}\left(
u^{1}-k\right) ^{p}\,dx+2^{p-1}L\,\left\vert k\right\vert ^{p}\,\mathcal{L}%
^{n}\left( A_{k,s}^{1}\right) $ \\ 
$+2m\left( L+1\right) \left[ \mathcal{L}^{n}\left( A_{k,s}^{1}\right) \right]
^{1-\frac{q}{p}}\left\Vert u\right\Vert _{W^{1,p}\left( A_{k,s}^{1}\right)
}^{q}+\left( L+1\right) \left[ \mathcal{L}^{n}\left( A_{k,s}^{1}\right) %
\right] ^{1-\frac{1}{\sigma }}\left\Vert a\right\Vert _{L^{\sigma }\left(
A_{k,s}^{1}\right) }$%
\end{tabular}%
\end{equation*}%
Moreover, since $1\leq q<\frac{p^{2}}{n}$ and $\sigma >\frac{n}{p}$, we get%
\begin{equation*}
\begin{tabular}{l}
$\int\limits_{A_{k,s}^{1}}\left\vert \nabla u^{1}\right\vert
^{p}\,+\left\vert \nabla u^{1}\right\vert ^{q}\,dx$ \\ 
$\leq L\int\limits_{A_{k,s}^{1}\backslash A_{k,t}^{1}}\left\vert \nabla
u^{1}\right\vert ^{p}+\left\vert \nabla u^{1}\right\vert ^{q}\,dx+\frac{%
2^{p}\left( p^{p}+p^{q}\right) }{\left( s-t\right) ^{p}}\int%
\limits_{A_{k,s}^{1}\backslash A_{k,t}^{1}}\left( u^{1}-k\right) ^{p}\,dx$
\\ 
$+D_{\Sigma }\left( 1+\left\vert k\right\vert ^{p}\right) \left[ \mathcal{L}%
^{n}\left( A_{k,s}^{1}\right) \right] ^{1-\frac{p}{n}+\varepsilon }$%
\end{tabular}%
\end{equation*}%
where%
\begin{equation*}
D_{\Sigma }=2^{p-1}L\left( p^{q}+L+2m\left( L+1\right) \left\Vert
u\right\Vert _{W^{1,p}\left( \Sigma \right) }^{q}+\left( L+1\right)
\left\Vert a\right\Vert _{L^{\sigma }\left( \Sigma \right) }\right)
\end{equation*}%
then it follows%
\begin{equation*}
\begin{tabular}{l}
$\int\limits_{A_{k,s}^{1}}\left\vert \nabla u^{1}\right\vert
^{p}\,+\left\vert \nabla u^{1}\right\vert ^{q}\,dx$ \\ 
$\leq \frac{L}{L+1}\int\limits_{A_{k,s}^{1}}\left\vert \nabla
u^{1}\right\vert ^{p}+\left\vert \nabla u^{1}\right\vert ^{q}\,dx+\frac{%
2^{p}\left( p^{p}+p^{q}\right) }{\left( L+1\right) \left( s-t\right) ^{p}}%
\int\limits_{A_{k,s}^{1}}\left( u^{1}-k\right) ^{p}\,dx$ \\ 
$+\frac{D_{\Sigma }}{L+1}\left( 1+\left\vert k\right\vert ^{p}\right) \left[ 
\mathcal{L}^{n}\left( A_{k,s}^{1}\right) \right] ^{1-\frac{p}{n}+\varepsilon
}$%
\end{tabular}%
\end{equation*}%
Using Lemma \ref{lemma3} we have the Caccioppoli Inequaity%
\begin{equation*}
\begin{tabular}{l}
$\int\limits_{A_{k,\varrho }^{1}}\left\vert \nabla u^{1}\right\vert
^{p}\,+\left\vert \nabla u^{1}\right\vert ^{q}\,dx$ \\ 
$\leq \frac{C_{1,\Sigma }}{\left( R-\varrho \right) ^{p}}\int%
\limits_{A_{k,R}^{1}}\left( u^{1}-k\right) ^{p}\,dx+C_{2,\Sigma }\left(
1+R^{-\varepsilon n}\left\vert k\right\vert ^{p}\right) \left[ \mathcal{L}%
^{n}\left( A_{k,R}^{1}\right) \right] ^{1-\frac{p}{n}+\varepsilon }$%
\end{tabular}%
\end{equation*}%
Similarly you can proceed for $\alpha =2,...,m$ and we get $u^{\alpha }\in
DG^{+}\left( \Omega ,p,\lambda ,\lambda _{\ast },\chi ,\varepsilon
,R_{0}\right) $ for every $\alpha =1,...,m$, with $\lambda =C_{\alpha
,\Sigma }$, $\lambda _{\ast }=C_{\alpha ,\Sigma }$ and $\chi =1$. Since $-u$
is a minimizer of the integral functional 
\begin{equation*}
\mathcal{\tilde{F}}\left( v,\Omega \right) =\int\limits_{\Omega
}\sum\limits_{\alpha =1}^{m}\left\vert \nabla v^{\alpha }\right\vert
^{p}+G\left( x,-v,\left\vert \nabla v^{1}\right\vert ,...,\left\vert \nabla
v^{m}\right\vert \right) \,dx
\end{equation*}%
then $u^{\alpha }\in DG^{-}\left( \Omega ,p,\lambda ,\lambda _{\ast },\chi
,\varepsilon ,R_{0}\right) $ for every $\alpha =1,...,m$, with $\lambda
=C_{\alpha ,\Sigma }$, $\lambda _{\ast }=C_{\alpha ,\Sigma }$ and $\chi =1$.
It follows that $u^{\alpha }\in DG\left( \Omega ,p,\lambda ,\lambda _{\ast
},\chi ,\varepsilon ,R_{0}\right) $ for every $\alpha =1,...,m$, with $%
\lambda =C_{\alpha ,\Sigma }$, $\lambda _{\ast }=C_{\alpha ,\Sigma }$ and $%
\chi =1$.

\bigskip

\section{Proof of Theorem \protect\ref{th1}}

Now, using the techniques introduced in [12, 22, 24, 31, 32, 33], the proof
of Theorem \ref{th1} is a direct consequence of Theorem \ref{th3} and
Theorem \ref{th2}.

\begin{theorem}
If $u\in W^{1,p}\left( \Omega ,%
%TCIMACRO{\U{211d} }%
%BeginExpansion
\mathbb{R}
%EndExpansion
^{m}\right) $ is a minimum of the functional (1.1) then $u\in W^{1,p}\left(
\Omega ,%
%TCIMACRO{\U{211d} }%
%BeginExpansion
\mathbb{R}
%EndExpansion
^{m}\right) \cap L_{loc}^{\infty }\left( \Omega ,%
%TCIMACRO{\U{211d} }%
%BeginExpansion
\mathbb{R}
%EndExpansion
^{m}\right) $.
\end{theorem}

\begin{proof}
It follows using Theorem \ref{th3} and Theorem \ref{th2}.
\end{proof}

In particular, if $\Sigma \subset \subset \Omega $ is a compact subset of $%
\Omega $, it follows that 
\begin{equation}
\left\vert u\right\vert <M=\sqrt{\sum\limits_{\alpha =1}^{m}\left( M^{\alpha
}\right) ^{2}}
\end{equation}%
where $M^{\alpha }=\sup\limits_{\Sigma }\left\{ \left\vert u^{\alpha
}\right\vert \right\} $. We observe that from the hypothesis H.1 it follows
that for every $\beta \in \left[ 1,....,m\right] $ we get 
\begin{equation}
\sum\limits_{\alpha =1}^{m}\left\vert \xi ^{\alpha }\right\vert
^{q}-\sum\limits_{\alpha =1}^{m}\left\vert M^{\alpha }\right\vert
^{q}-a\left( x\right) \leq G\left( x,u^{1},...,u^{m},\left\vert \xi
^{1}\right\vert ,...,\left\vert \xi ^{m}\right\vert \right) \leq L\left[
\sum\limits_{\alpha =1}^{m}\left\vert \xi ^{\alpha }\right\vert
^{q}+\left\vert u^{\beta }\right\vert ^{q}+\sum\limits_{\alpha \neq \beta
}^{{}}\left\vert M^{\alpha }\right\vert ^{q}+a\left( x\right) \right]
\end{equation}%
for $\mathcal{L}^{n}$ a. e. $x\in \Omega $, for every $u^{\alpha }\in 
%TCIMACRO{\U{211d} }%
%BeginExpansion
\mathbb{R}
%EndExpansion
$, $\left\vert u^{\alpha }\right\vert <M^{\alpha }$ and for every $\xi
^{\alpha }\in 
%TCIMACRO{\U{211d} }%
%BeginExpansion
\mathbb{R}
%EndExpansion
$ with $\alpha =1,...,m$ and $m\geq 1$ and with $a\left( x\right) \in
L^{\sigma }\left( \Omega \right) $, $a(x)\geq 0$ for $\mathcal{L}^{n}$ a. e. 
$x\in \Omega $, $\sigma >\frac{n}{p}$, $1\leq q<\frac{p^{2}}{n}$ and $1<p<n$.

\begin{theorem}
Let $u\in W^{1,p}\left( \Omega ,%
%TCIMACRO{\U{211d} }%
%BeginExpansion
\mathbb{R}
%EndExpansion
^{m}\right) \cap L_{loc}^{\infty }\left( \Omega ,%
%TCIMACRO{\U{211d} }%
%BeginExpansion
\mathbb{R}
%EndExpansion
^{m}\right) $ be a minimun of the functional (1.1)\ then for all $\Sigma
\subset \subset \Omega $\ compact subset of $\Omega $ there exist two
positive constants $\tilde{D}_{1,\Sigma }$ and $\tilde{D}_{2,\Sigma }$ such
that for all $x_{0}\in \Sigma $, for every $0<\varrho <R<\min \left\{ 1,%
\frac{dist\left( x_{0},\partial \Sigma \right) }{2}\right\} $ and for every $%
k\in 
%TCIMACRO{\U{211d} }%
%BeginExpansion
\mathbb{R}
%EndExpansion
$ it follows 
\begin{equation}
\begin{tabular}{l}
$\int\limits_{A_{k,\varrho }^{\alpha }}\left\vert \nabla u^{\alpha
}\right\vert ^{p}\,+\left\vert \nabla u^{\alpha }\right\vert ^{q}\,dx$ \\ 
$\leq \frac{\tilde{D}_{1,\Sigma }}{\left( R-\varrho \right) ^{p}}%
\int\limits_{A_{k,R}^{\alpha }}\left( u^{\alpha }-k\right) ^{p}\,dx+\tilde{D}%
_{2,\Sigma }\left[ \mathcal{L}^{n}\left( A_{k,R}^{\alpha }\right) \right]
^{1-\frac{p}{n}+\varepsilon }$%
\end{tabular}%
\end{equation}%
and%
\begin{equation}
\begin{tabular}{l}
$\int\limits_{B_{k,\varrho }^{\alpha }}\left\vert \nabla u^{\alpha
}\right\vert ^{p}\,+\left\vert \nabla u^{\alpha }\right\vert ^{q}\,dx$ \\ 
$\leq \frac{\tilde{D}_{1,\Sigma }}{\left( R-\varrho \right) ^{p}}%
\int\limits_{B_{k,R}^{\alpha }}\left( k-u^{\alpha }\right) ^{p}\,dx+\tilde{D}%
_{2,\Sigma }\left[ \mathcal{L}^{n}\left( B_{k,R}^{\alpha }\right) \right]
^{1-\frac{p}{n}+\varepsilon }$%
\end{tabular}%
\end{equation}%
for every $\alpha =1,...,m$.
\end{theorem}

\begin{proof}
Let $\Sigma \subset \subset \Omega $ be a compact subset of $\Omega $, let $%
R_{0}$ be a positive real number, let us consider $x_{0}\in \Sigma $, $%
0<\varrho \leq t<s\leq R<\min \left( 1,\frac{dist\left( x_{0},\partial
\Sigma \right) }{2}\right) $, $k\in 
%TCIMACRO{\U{211d} }%
%BeginExpansion
\mathbb{R}
%EndExpansion
$\ and $\eta \in C_{c}^{\infty }\left( B_{s}\left( x_{0}\right) \right) $
such that $0\leq \eta \leq 1$ on $B_{s}\left( x_{0}\right) $, $\eta =1$ in $%
B_{t}\left( x_{0}\right) $ and $\left\vert \nabla \eta \right\vert \leq 
\frac{2}{s-t}$ on $B_{s}\left( x_{0}\right) $, moreover we we can observe
that supp$\left( \nabla \eta \right) \subset B_{s}\left( x_{0}\right)
\backslash B_{t}\left( x_{0}\right) $. If $u_{\varepsilon }$ is a local
sub-minimum of $J_{\varepsilon }\left( u,\Omega \right) $ we can define $%
\varphi =-\eta ^{p}w_{+}$ where%
\begin{equation*}
w_{+}=\left( 
\begin{tabular}{l}
$\left( u^{1}-k\right) _{+}$ \\ 
$0$ \\ 
$\vdots $ \\ 
$0$%
\end{tabular}%
\right)
\end{equation*}%
with $\left( u^{1}-k\right) _{+}=\max \left\{ u^{1}-k,0\right\} $ and,
moreover, we get%
\begin{equation*}
\mathcal{F}\left( u,B_{s}\left( x_{0}\right) \right) \leq \mathcal{F}\left(
u+\varphi ,B_{s}\left( x_{0}\right) \right)
\end{equation*}%
\ The previous variational inequality can be written like this%
\begin{equation*}
\begin{tabular}{l}
$\int\limits_{A_{k,s}^{1}}\left\vert \nabla u^{1}\right\vert
^{p}\,dx+\int\limits_{A_{k,s}^{1}}\sum\limits_{\alpha =2}^{m}\left\vert
\nabla u^{\alpha }\right\vert ^{p}\,dx+\int\limits_{B_{s}\left( x_{0}\right)
\backslash A_{k,s}^{1}}\sum\limits_{\alpha =1}^{m}\left\vert \nabla
u^{\alpha }\right\vert ^{p}\,dx+$ \\ 
$\int\limits_{A_{k,s}^{1}}G\left( x,u^{1},...,u^{m},\left\vert \nabla
u^{1}\right\vert ,...,\left\vert \nabla u^{m}\right\vert \right)
\,dx+\int\limits_{B_{s}\left( x_{0}\right) \backslash A_{k,s}^{1}}G\left(
x,u,\left\vert \nabla u^{1}\right\vert ,...,\left\vert \nabla
u^{m}\right\vert \right) \,dx$ \\ 
$\leq \int\limits_{A_{k,s}^{1}}\left( 1-\eta ^{p}\right) ^{p}\left\vert
\nabla u^{1}\right\vert ^{p}\,dx+p^{p}\int\limits_{A_{k,s}^{1}\backslash
A_{k,t}^{1}}\eta ^{\left( p-1\right) p}\left\vert \nabla \eta \right\vert
^{p}\left( u^{1}-k\right) ^{p}\,dx$ \\ 
$+\int\limits_{A_{k,s}^{1}}\sum\limits_{\alpha =2}^{m}\left\vert \nabla
u^{\alpha }\right\vert ^{p}\,dx+\int\limits_{B_{s}\left( x_{0}\right)
\backslash A_{k,s}^{1}}\sum\limits_{\alpha =1}^{m}\left\vert \nabla
u^{\alpha }\right\vert ^{p}\,dx$ \\ 
$+\int\limits_{A_{k,s}^{1}}G\left( x,\left( 1-\eta ^{p}\right) \left(
u^{1}-k\right) +k,...,u^{m},\left\vert \left( 1-\eta ^{p}\right) \nabla
u^{1}+p\eta ^{p-1}\nabla \eta \left( u^{1}-k\right) \right\vert
,...,\left\vert \nabla u^{m}\right\vert \right) \,dx$ \\ 
$+\int\limits_{B_{s}\left( x_{0}\right) \backslash A_{k,s}^{1}}G\left(
x,u,\left\vert \nabla u^{1}\right\vert ,...,\left\vert \nabla
u^{m}\right\vert \right) \,dx$%
\end{tabular}%
\end{equation*}%
where $A_{k,\varrho }^{1}=\left\{ u^{1}>k\right\} \cap B_{\varrho }\left(
x_{0}\right) $\ and with simple algebraic calculations then we get%
\begin{equation*}
\begin{tabular}{l}
$\int\limits_{A_{k,s}^{1}}\left\vert \nabla u^{1}\right\vert
^{p}\,dx+\int\limits_{A_{k,s}^{1}}G\left( x,u^{1},...,u^{m},\left\vert
\nabla u^{1}\right\vert ,...,\left\vert \nabla u^{m}\right\vert \right) \,dx$
\\ 
$\leq \int\limits_{A_{k,s}^{1}}\left( 1-\eta ^{p}\right) \left\vert \nabla
u^{1}\right\vert \,dx+p^{p}\int\limits_{A_{k,s}^{1}\backslash
A_{k,t}^{1}}\eta ^{p-1}\left\vert \nabla \eta \right\vert ^{p}\left(
u^{1}-k\right) ^{p}\,dx$ \\ 
$+\int\limits_{A_{k,s}^{1}}G\left( x,\left( 1-\eta ^{p}\right) \left(
u^{1}-k\right) +k,...,u^{m},\left\vert \left( 1-\eta ^{p}\right) \nabla
u^{1}+p\eta ^{p-1}\nabla \eta \left( u^{1}-k\right) \right\vert
,...,\left\vert \nabla u^{m}\right\vert \right) \,dx$%
\end{tabular}%
\end{equation*}%
Recalling that from (5.1), with $\beta =1$, we have%
\begin{equation*}
\sum\limits_{\alpha =1}^{m}\left\vert \nabla u^{\alpha }\right\vert
^{q}-\sum\limits_{\alpha =1}^{m}\left\vert M^{\alpha }\right\vert
^{q}-a\left( x\right) \leq G\left( x,u^{1},...,u^{m},\left\vert \nabla
u^{1}\right\vert ,...,\left\vert \nabla u^{m}\right\vert \right)
\end{equation*}%
and%
\begin{equation*}
\begin{tabular}{l}
$G\left( x,\left( 1-\eta ^{p}\right) \left( u^{1}-k\right)
+k,...,u^{m},\left\vert \left( 1-\eta ^{p}\right) \nabla u^{1}+p\eta
^{p-1}\nabla \eta \left( u^{1}-k\right) \right\vert ,...,\left\vert \nabla
u^{m}\right\vert \right) $ \\ 
$\leq L\left\vert \left( 1-\eta ^{p}\right) \nabla u^{1}+p\eta ^{p-1}\nabla
\eta \left( u^{1}-k\right) \right\vert ^{q}+L\sum\limits_{\alpha
=2}^{m}\left\vert \nabla u^{\alpha }\right\vert ^{q}$ \\ 
$+L\left\vert u^{1}-\eta ^{p}\left( u^{1}-k\right) \right\vert
^{q}+L\sum\limits_{\alpha =2}^{m}\left\vert M^{\alpha }\right\vert
^{q}+La\left( x\right) $%
\end{tabular}%
\end{equation*}%
then we get%
\begin{equation*}
\begin{tabular}{l}
$\int\limits_{A_{k,s}^{1}}\left\vert \nabla u^{1}\right\vert
^{p}\,dx+\int\limits_{A_{k,s}^{1}}\sum\limits_{\alpha =1}^{m}\left\vert
\nabla u^{\alpha }\right\vert ^{q}-\sum\limits_{\alpha =1}^{m}\left\vert
M^{\alpha }\right\vert ^{q}-a\left( x\right) \,dx$ \\ 
$\leq \int\limits_{A_{k,s}^{1}}\left( 1-\eta ^{p}\right) \left\vert \nabla
u^{1}\right\vert ^{p}\,dx+p^{p}\int\limits_{A_{k,s}^{1}\backslash
A_{k,t}^{1}}\eta ^{p-1}\left\vert \nabla \eta \right\vert ^{p}\left(
u^{1}-k\right) ^{p}\,dx$ \\ 
$+\int\limits_{A_{k,s}^{1}}L\left\vert \left( 1-\eta ^{p}\right) \nabla
u^{1}+p\eta ^{p-1}\nabla \eta \left( u^{1}-k\right) \right\vert
^{q}+L\sum\limits_{\alpha =2}^{m}\left\vert \nabla u^{\alpha }\right\vert
^{q}\,dx$ \\ 
$+\int\limits_{A_{k,s}^{1}}L\left\vert u^{1}-\eta ^{p}\left( u^{1}-k\right)
\right\vert ^{q}+L\sum\limits_{\alpha =2}^{m}\left\vert M^{\alpha
}\right\vert ^{q}+La\left( x\right) \,dx$%
\end{tabular}%
\end{equation*}%
\bigskip \bigskip From which, with simple weighting we have%
\begin{equation*}
\begin{tabular}{l}
$\int\limits_{A_{k,s}^{1}}\left\vert \nabla u^{1}\right\vert
^{p}\,dx+\int\limits_{A_{k,s}^{1}}\sum\limits_{\alpha =1}^{m}\left\vert
\nabla u^{\alpha }\right\vert ^{q}\,dx$ \\ 
$\leq \int\limits_{A_{k,s}^{1}}\left( 1-\eta ^{p}\right) \left\vert \nabla
u^{1}\right\vert ^{p}\,dx+p^{p}\int\limits_{A_{k,s}^{1}\backslash
A_{k,t}^{1}}\eta ^{p-1}\left\vert \nabla \eta \right\vert ^{p}\left(
u^{1}-k\right) ^{p}\,dx$ \\ 
$+\int\limits_{A_{k,s}^{1}}\sum\limits_{\alpha =1}^{m}\left\vert M^{\alpha
}\right\vert ^{q}+a\left( x\right) \,dx$ \\ 
$+\int\limits_{A_{k,s}^{1}}L\left( 1-\eta ^{p}\right) ^{q}\left\vert \nabla
u^{1}\right\vert ^{q}+p^{q}\eta ^{\left( p-1\right) q}\left\vert \nabla \eta
\right\vert ^{q}\left( u^{1}-k\right) ^{q}+L\sum\limits_{\alpha
=2}^{m}\left\vert \nabla u^{\alpha }\right\vert ^{q}\,dx$ \\ 
$+\int\limits_{A_{k,s}^{1}}L\left\vert u^{1}-\eta ^{p}\left( u^{1}-k\right)
\right\vert ^{q}+L\sum\limits_{\alpha =2}^{m}\left\vert M^{\alpha
}\right\vert ^{q}+La\left( x\right) \,dx$%
\end{tabular}%
\end{equation*}%
Since $1-\eta ^{p}=0$ on $A_{k,t}^{1}$, it follows that%
\begin{equation}
\begin{tabular}{l}
$\ \int\limits_{A_{k,s}^{1}}\left\vert \nabla u^{1}\right\vert
^{p}\,+\left\vert \nabla u^{1}\right\vert ^{q}\,dx$ \\ 
$\leq L\int\limits_{A_{k,s}^{1}\backslash A_{k,t}^{1}}\left\vert \nabla
u^{1}\right\vert ^{p}+\left\vert \nabla u^{1}\right\vert
^{q}\,dx+p^{p}\int\limits_{A_{k,s}^{1}\backslash A_{k,t}^{1}}\eta
^{p-1}\left\vert \nabla \eta \right\vert ^{p}\left( u^{1}-k\right) ^{p}\,dx$
\\ 
$+\int\limits_{A_{k,s}^{1}}\sum\limits_{\alpha =1}^{m}\left\vert M^{\alpha
}\right\vert ^{q}\,dx+\int\limits_{A_{k,s}^{1}}L\left\vert u^{1}-\eta
^{p}\left( u^{1}-k\right) \right\vert ^{q}+L\sum\limits_{\alpha
=2}^{m}\left\vert M^{\alpha }\right\vert ^{q}+\left( L+1\right) a\left(
x\right) \,dx$ \\ 
$+\int\limits_{A_{k,s}^{1}}p^{q}\eta ^{\left( p-1\right) q}\left\vert \nabla
\eta \right\vert ^{q}\left( u^{1}-k\right) ^{q}+L\sum\limits_{\alpha
=2}^{m}\left\vert \nabla u^{\alpha }\right\vert ^{q}\,dx$%
\end{tabular}%
\end{equation}%
We observe that we must estimate the term 
\begin{equation*}
\int\limits_{A_{k,s}^{1}}L\left\vert u^{1}-\eta ^{p}\left( u^{1}-k\right)
\right\vert ^{q}\,dx\leq 2^{q-1}L\int\limits_{A_{k,s}^{1}}\left\vert
u^{1}\right\vert ^{q}+\left\vert \eta ^{p}\left( u^{1}-k\right) \right\vert
^{q}\,dx
\end{equation*}%
then, since $\left\vert u^{1}\right\vert <M^{1}$ in $\Sigma $, it follows 
\begin{equation*}
\int\limits_{A_{k,s}^{1}}L\left\vert u^{1}-\eta ^{p}\left( u^{1}-k\right)
\right\vert ^{q}\,dx\leq 2^{q-1}L\int\limits_{A_{k,s}^{1}}\left\vert
M^{1}\right\vert ^{q}+\left\vert \eta ^{p}\left( u^{1}-k\right) \right\vert
^{q}\,dx
\end{equation*}%
using Young inelality we get%
\begin{equation}
\int\limits_{A_{k,s}^{1}}L\left\vert u^{1}-\eta ^{p}\left( u^{1}-k\right)
\right\vert ^{q}\,dx\leq 2^{q-1}L\int\limits_{A_{k,s}^{1}}\left\vert
M^{1}\right\vert ^{q}+\frac{p-q}{p}\eta ^{p}\left( s-t\right) ^{\frac{p}{p-q}%
}+\frac{q}{p}\frac{\eta ^{p}\left( u^{1}-k\right) ^{p}}{\left( s-t\right)
^{p}}\,dx
\end{equation}%
Using (5.3) and (5.4) it follows%
\begin{equation}
\begin{tabular}{l}
$\ \int\limits_{A_{k,s}^{1}}\left\vert \nabla u^{1}\right\vert
^{p}\,+\left\vert \nabla u^{1}\right\vert ^{q}\,dx$ \\ 
$\leq L\int\limits_{A_{k,s}^{1}\backslash A_{k,t}^{1}}\left\vert \nabla
u^{1}\right\vert ^{p}+\left\vert \nabla u^{1}\right\vert
^{q}\,dx+p^{p}\int\limits_{A_{k,s}^{1}\backslash A_{k,t}^{1}}\eta
^{p-1}\left\vert \nabla \eta \right\vert ^{p}\left( u^{1}-k\right) ^{p}\,dx$
\\ 
$+\int\limits_{A_{k,s}^{1}}\left( 1+2^{q-1}L\right) \sum\limits_{\alpha
=1}^{m}\left\vert M^{\alpha }\right\vert ^{q}+\left( L+1\right) a\left(
x\right) \,dx$ \\ 
$+2^{q-1}L\int\limits_{A_{k,s}^{1}}\frac{p-q}{p}\eta ^{p}\left( s-t\right) ^{%
\frac{p}{p-q}}+\frac{q}{p}\frac{\eta ^{p}\left( u^{1}-k\right) ^{p}}{\left(
s-t\right) ^{p}}\,dx$ \\ 
$+\int\limits_{A_{k,s}^{1}}p^{q}\eta ^{\left( p-1\right) q}\left\vert \nabla
\eta \right\vert ^{q}\left( u^{1}-k\right) ^{q}+L\sum\limits_{\alpha
=2}^{m}\left\vert \nabla u^{\alpha }\right\vert ^{q}\,dx$%
\end{tabular}%
\end{equation}%
We observe that we must estimate the term 
\begin{equation*}
\left( L+1\right) \int\limits_{A_{k,s}^{1}}a\left( x\right) \,dx
\end{equation*}%
using the H\"{o}lder inequality \ref{4.2} obtaining%
\begin{equation}
\begin{tabular}{l}
$\left( L+1\right) \int\limits_{A_{k,s}^{1}}a\left( x\right) \,dx$ \\ 
$\left( L+1\right) \left[ \mathcal{L}^{n}\left( A_{k,s}^{1}\right) \right]
^{1-\frac{1}{\sigma }}\left[ \int\limits_{A_{k,s}^{1}}a^{\sigma }\,dx\right]
^{\frac{1}{\sigma }}$%
\end{tabular}%
\end{equation}%
Similarly we can also estimate the term%
\begin{equation*}
\int\limits_{A_{k,s}^{1}}L\sum\limits_{\alpha =2}^{m}\left\vert \nabla
u^{\alpha }\right\vert ^{q}\,dx
\end{equation*}%
using the H\"{o}lder inequality \ref{4.2} obtaining%
\begin{equation}
\begin{tabular}{l}
$\int\limits_{A_{k,s}^{1}}L\sum\limits_{\alpha =2}^{m}\left\vert \nabla
u^{\alpha }\right\vert ^{q}\,dx$ \\ 
$\leq mL\left[ \mathcal{L}^{n}\left( A_{k,s}^{1}\right) \right] ^{1-\frac{q}{%
p}}\left[ \int\limits_{A_{k,s}^{1}}\left\vert \nabla u\right\vert ^{p}\,dx%
\right] ^{\frac{q}{p}}$%
\end{tabular}%
\end{equation}%
Using(5.5), (5.6) and (5.7)\ we get%
\begin{equation}
\begin{tabular}{l}
$\ \int\limits_{A_{k,s}^{1}}\left\vert \nabla u^{1}\right\vert
^{p}\,+\left\vert \nabla u^{1}\right\vert ^{q}\,dx$ \\ 
$\leq L\int\limits_{A_{k,s}^{1}\backslash A_{k,t}^{1}}\left\vert \nabla
u^{1}\right\vert ^{p}+\left\vert \nabla u^{1}\right\vert
^{q}\,dx+p^{p}\int\limits_{A_{k,s}^{1}\backslash A_{k,t}^{1}}\eta
^{p-1}\left\vert \nabla \eta \right\vert ^{p}\left( u^{1}-k\right) ^{p}\,dx$
\\ 
$+2^{q-1}L\int\limits_{A_{k,s}^{1}}\frac{p-q}{p}\eta ^{p}\left( s-t\right) ^{%
\frac{p}{p-q}}+\frac{q}{p}\frac{\eta ^{p}\left( u^{1}-k\right) ^{p}}{\left(
s-t\right) ^{p}}\,dx+\int\limits_{A_{k,s}^{1}}p^{q}\eta ^{\left( p-1\right)
q}\left\vert \nabla \eta \right\vert ^{q}\left( u^{1}-k\right) ^{q}\,dx$ \\ 
$+\tilde{L}\left( q,m,M,L\right) \mathcal{L}^{n}\left( A_{k,s}^{1}\right)
+2m\left( L+1\right) \left[ \mathcal{L}^{n}\left( A_{k,s}^{1}\right) \right]
^{1-\frac{q}{p}}\left\Vert u\right\Vert _{W^{1,p}\left( A_{k,s}^{1}\right)
}^{q}+\left( L+1\right) \left[ \mathcal{L}^{n}\left( A_{k,s}^{1}\right) %
\right] ^{1-\frac{1}{\sigma }}\left\Vert a\right\Vert _{L^{\sigma }\left(
A_{k,s}^{1}\right) }$%
\end{tabular}%
\end{equation}%
where $\tilde{L}\left( q,m,M,L\right) =m\left( 1+2^{q-1}L\right) \left(
1+\left\vert M\right\vert \right) ^{q}$. Using the Young Inequality \ref{4.1}
we have%
\begin{equation}
\begin{tabular}{l}
$\ \int\limits_{A_{k,s}^{1}}p^{q}\eta ^{\left( p-1\right) q}\left\vert
\nabla \eta \right\vert ^{q}\left( u^{1}-k\right) ^{q}\,dx$ \\ 
$=\int\limits_{A_{k,s}^{1}\backslash A_{k,t}^{1}}p^{q}\eta ^{\left(
p-1\right) q}\left\vert \nabla \eta \right\vert ^{q}\left( u^{1}-k\right)
^{q}\,dx$ \\ 
$\leq p^{q}\int\limits_{A_{k,s}^{1}\backslash A_{k,t}^{1}}1+\left\vert
\nabla \eta \right\vert ^{p}\left( u^{1}-k\right) ^{p}\,dx$ \\ 
$\leq p^{q}\mathcal{L}^{n}\left( A_{k,s}^{1}\right)
+p^{q}\int\limits_{A_{k,s}^{1}\backslash A_{k,t}^{1}}\left\vert \nabla \eta
\right\vert ^{p}\left( u^{1}-k\right) ^{p}\,dx$%
\end{tabular}%
\end{equation}%
then using (5.8) and (5.9), since $0\leq \eta \leq 1$, it follows%
\begin{equation*}
\begin{tabular}{l}
$\ \int\limits_{A_{k,s}^{1}}\left\vert \nabla u^{1}\right\vert
^{p}\,+\left\vert \nabla u^{1}\right\vert ^{q}\,dx$ \\ 
$\leq L\int\limits_{A_{k,s}^{1}\backslash A_{k,t}^{1}}\left\vert \nabla
u^{1}\right\vert ^{p}+\left\vert \nabla u^{1}\right\vert ^{q}\,dx+\left(
p^{p}+p^{q}\right) \int\limits_{A_{k,s}^{1}\backslash A_{k,t}^{1}}\left\vert
\nabla \eta \right\vert ^{p}\left( u^{1}-k\right) ^{p}\,dx$ \\ 
$+2^{q-1}L\int\limits_{A_{k,s}^{1}}\frac{p-q}{p}\eta ^{p}\left( s-t\right) ^{%
\frac{p}{p-q}}+\frac{q}{p}\frac{\eta ^{p}\left( u^{1}-k\right) ^{p}}{\left(
s-t\right) ^{p}}\,dx+\left( p^{q}+\tilde{L}\left( q,m,M,L\right) \right) 
\mathcal{L}^{n}\left( A_{k,s}^{1}\right) $ \\ 
$+2m\left( L+1\right) \left[ \mathcal{L}^{n}\left( A_{k,s}^{1}\right) \right]
^{1-\frac{q}{p}}\left\Vert u\right\Vert _{W^{1,p}\left( A_{k,s}^{1}\right)
}^{q}+\left( L+1\right) \left[ \mathcal{L}^{n}\left( A_{k,s}^{1}\right) %
\right] ^{1-\frac{1}{\sigma }}\left\Vert a\right\Vert _{L^{\sigma }\left(
A_{k,s}^{1}\right) }$%
\end{tabular}%
\end{equation*}%
Moreover, since $0<\left( s-t\right) <R<1$, $\left\vert \nabla \eta
\right\vert ^{p}\leq \frac{2^{p}}{\left( s-t\right) ^{p}}$, $1\leq q<\frac{%
p^{2}}{n}$ and $\sigma >\frac{n}{p}$, we get%
\begin{equation*}
\begin{tabular}{l}
$\int\limits_{A_{k,s}^{1}}\left\vert \nabla u^{1}\right\vert
^{p}\,+\left\vert \nabla u^{1}\right\vert ^{q}\,dx$ \\ 
$\leq L\int\limits_{A_{k,s}^{1}\backslash A_{k,t}^{1}}\left\vert \nabla
u^{1}\right\vert ^{p}+\left\vert \nabla u^{1}\right\vert ^{q}\,dx+\frac{D_{1}%
}{\left( s-t\right) ^{p}}\int\limits_{A_{k,s}^{1}}\left( u^{1}-k\right)
^{p}\,dx+$ \\ 
$D_{\Sigma }\left[ \mathcal{L}^{n}\left( A_{k,s}^{1}\right) \right] ^{1-%
\frac{p}{n}+\varepsilon }$%
\end{tabular}%
\end{equation*}%
where%
\begin{equation*}
D_{1}=\left( \left( p^{p}+p^{q}\right) 2^{p}+\frac{q2^{q-1}L}{p}\right)
\end{equation*}%
\begin{equation*}
D_{\Sigma }=\left( \left( 2^{q-1}L\frac{p-q}{p}+p^{q}+\tilde{L}\left(
q,m,M,L\right) \right) \left( \varpi _{n}\right) ^{\frac{p}{n}-\varepsilon
}+2m\left( L+1\right) \left\Vert u\right\Vert _{W^{1,p}\left( \Sigma \right)
}^{q}+\left( L+1\right) \left\Vert a\right\Vert _{L^{\sigma }\left( \Sigma
\right) }\right)
\end{equation*}%
with $\varpi _{n}=\mathcal{L}^{n}\left( B_{1}\left( 0\right) \right) $.
Moreover, it follows%
\begin{equation*}
\begin{tabular}{l}
$\int\limits_{A_{k,s}^{1}}\left\vert \nabla u^{1}\right\vert
^{p}\,+\left\vert \nabla u^{1}\right\vert ^{q}\,dx$ \\ 
$\leq \frac{L}{L+1}\int\limits_{A_{k,s}^{1}}\left\vert \nabla
u^{1}\right\vert ^{p}+\left\vert \nabla u^{1}\right\vert ^{q}\,dx+\frac{D_{1}%
}{\left( L+1\right) \left( s-t\right) ^{p}}\int\limits_{A_{k,s}^{1}}\left(
u^{1}-k\right) ^{p}\,dx+$ \\ 
$\frac{D_{\Sigma }}{L+1}\left[ \mathcal{L}^{n}\left( A_{k,s}^{1}\right) %
\right] ^{1-\frac{p}{n}+\varepsilon }$%
\end{tabular}%
\end{equation*}%
Using Lemma \ref{lemma3} we have the Caccioppoli Inequaity%
\begin{equation*}
\begin{tabular}{l}
$\int\limits_{A_{k,\varrho }^{1}}\left\vert \nabla u^{1}\right\vert
^{p}\,+\left\vert \nabla u^{1}\right\vert ^{q}\,dx$ \\ 
$\leq \frac{C_{1,\Sigma }}{\left( R-\varrho \right) ^{p}}\int%
\limits_{A_{k,R}^{1}}\left( u^{1}-k\right) ^{p}\,dx+C_{2,\Sigma }\left[ 
\mathcal{L}^{n}\left( A_{k,R}^{1}\right) \right] ^{1-\frac{p}{n}+\varepsilon
}$%
\end{tabular}%
\end{equation*}%
Similarly you can proceed for $\alpha =2,...,m$ and we get%
\begin{equation*}
\begin{tabular}{l}
$\int\limits_{A_{k,\varrho }^{\alpha }}\left\vert \nabla u^{\alpha
}\right\vert ^{p}\,+\left\vert \nabla u^{\alpha }\right\vert ^{q}\,dx$ \\ 
$\leq \frac{C_{1,\Sigma }}{\left( R-\varrho \right) ^{p}}\int%
\limits_{A_{k,R}^{\alpha }}\left( u^{\alpha }-k\right) ^{p}\,dx+C_{2,\Sigma }%
\left[ \mathcal{L}^{n}\left( A_{k,R}^{\alpha }\right) \right] ^{1-\frac{p}{n}%
+\varepsilon }$%
\end{tabular}%
\end{equation*}%
for every $\alpha =1,...,m$. Since $-u$ is a minimizer of the integral
functional 
\begin{equation*}
\mathcal{\tilde{F}}\left( v,\Omega \right) =\int\limits_{\Omega
}\sum\limits_{\alpha =1}^{m}\left\vert \nabla v^{\alpha }\right\vert
^{p}+G\left( x,-v,\left\vert \nabla v^{1}\right\vert ,...,\left\vert \nabla
v^{m}\right\vert \right) \,dx
\end{equation*}%
then 
\begin{equation*}
\begin{tabular}{l}
$\int\limits_{B_{k,\varrho }^{\alpha }}\left\vert \nabla u^{\alpha
}\right\vert ^{p}\,+\left\vert \nabla u^{\alpha }\right\vert ^{q}\,dx$ \\ 
$\leq \frac{C_{1,\Sigma }}{\left( R-\varrho \right) ^{p}}\int%
\limits_{B_{k,R}^{\alpha }}\left( k-u^{\alpha }\right) ^{p}\,dx+C_{2,\Sigma }%
\left[ \mathcal{L}^{n}\left( B_{k,R}^{\alpha }\right) \right] ^{1-\frac{p}{n}%
+\varepsilon }$%
\end{tabular}%
\end{equation*}%
for every $\alpha =1,...,m$.
\end{proof}

Now, Theorem \ref{th1} follows by applying Theorem 7 and Proposition 7.1,
Lemma 7.2 and Lemma 7.3 of [24].

\end{document}